\documentclass[11pt, reqno]{amsart}
\usepackage{hyperref}
\hypersetup{
    colorlinks=true,
    linkcolor=blue,
    filecolor=magenta,      
    urlcolor=cyan,
    citecolor=red
}
\usepackage{euscript}
\usepackage{amssymb}\usepackage{amscd}
\usepackage{tikz,underscore}
\usepackage{float}
\usepackage{graphicx}
\usepackage{caption}
\numberwithin{equation}{section}

\theoremstyle{plain}
\newtheorem{theorem}{Theorem}[section]
\newtheorem{lemma}[theorem]{Lemma}
\newtheorem{proposition}[theorem]{Proposition}

\newtheorem{corollary}[theorem]{Corollary}

\newtheorem{thm}[theorem]{Theorem}
\newtheorem{cor}[theorem]{Corollary}

\theoremstyle{definition}

\newtheorem{remark}[theorem]{Remark}
\newtheorem{rem}[theorem]{Remark}

\DeclareMathOperator{\Isom}{{\mathrm Isom}}
\DeclareMathOperator{\Hull}{{Hull}}

\def\H{\mathbb H}

\def\Z{\mathbb Z}

\def\TT{\mathcal{T}_\varepsilon}
\def\TH{\widehat{\mathcal{T}}_\varepsilon}

\def\F{\mathbb F}

\newcommand{\nbhd}[2]{\bar{N}_{#1}\left({#2}\right)}

\newcommand{\geo}{\partial X}

\makeatletter
\DeclareFontFamily{U}{tipa}{}
\DeclareFontShape{U}{tipa}{m}{n}{<->tipa10}{}
\newcommand{\arc@char}{{\usefont{U}{tipa}{m}{n}\symbol{62}}}%

\newcommand{\arc}[1]{\mathpalette\arc@arc{#1}}

\newcommand{\arc@arc}[2]{
  \sbox0{$\m@th#1#2$}
  \vbox{
    \hbox{\resizebox{\wd0}{\height}{\arc@char}}
    \nointerlineskip
    \box0
  }
}
\makeatother

\def\geo{\partial_{\infty}}
\def\Ga{\Gamma}
\def\al{\alpha}
\def\la{\lambda}

\def\ve{\varepsilon}
\def\dl{\delta}
\newcommand{\Bis}{\mathrm{Bis}}

\title{Ping-pong in Hadamard manifolds}

\author{Subhadip Dey}
\address{Department of Mathematics, UC Davis, One Shields Avenue, Davis, CA 95616, USA}
\email{sdey@math.ucdavis.edu}

\author{Michael Kapovich}
\address{Department of Mathematics, UC Davis, One Shields Avenue, Davis, CA 95616, USA}
\email{kapovich@math.ucdavis.edu}

\author{Beibei Liu}
\address{Department of Mathematics, UC Davis, One Shields Avenue, Davis, CA 95616, USA}
\email{bxliu@math.ucdavis.edu}

\date{\today}

\begin{document}
\maketitle

\begin{abstract}
In this paper, we prove a quantitative version of the Tits alternative for negatively pinched manifolds $X$. Precisely, we prove that a nonelementary discrete isometry subgroup of $\Isom(X)$ generated by two non-elliptic isometries $g,  f$ contains a free subgroup of rank 2 generated by isometries $f^N, h$ of uniformly bounded word length. Furthermore, we show that this free subgroup is convex-cocompact when $f$ is hyperbolic. 
\end{abstract}

\section{Introduction}
Let $X$ be an $n$-dimensional negatively curved Hadamard manifold, with sectional curvature ranging between $-\kappa^2$ and $1$, for some $\kappa \ge 1$. The main result of this note is the following quantitative version of the Tits alternative for $X$, which answers a question asked by Filippo  Cerocchi during the Oberwolfach Workshop ``Differentialgeometrie im Grossen'', 2017,  see also \cite{CS}: 

\begin{theorem}\label{thm:main}
There exists a function $\mathcal{L}=\mathcal{L}(n,\kappa)$ such that the following holds: 
Let $f, g$ be non-elliptic isometries of $X$ generating a nonelementary discrete subgroup $\Gamma$ of $\Isom(X)$. Then there exists an element $h\in \Gamma$ whose word length (with respect to the generators $f, g$) is $\le \mathcal{L}$ and a number $N\le \mathcal{L}$ such that the subgroup of $\Gamma$ generated by $f^N, h$ is free of rank two.  
\end{theorem}

One can regard this theorem as a quantitative version of the Tits alternative for discrete subgroups of $\Isom(X)$. For other forms of the quantitative Tits alternative we refer to \cite{BCG, B, BF, BG}.

\medskip 
After replacing $g$ with the element $g':=gfg^{-1}$, and noticing that the subgroup generated by $f, g'$ is still discrete and nonelementary, we reduce the problem to the case when the isometries $f$ and $g$ are conjugate in $\Isom(X)$ which we will assume from now on. 

\medskip 
The proof of  Theorem \ref{thm:main} breaks into two cases which are handled by different arguments:

{\bf Case 1.} $f$ (and, hence, $g$) has translation length bounded below by some positive number $\la$. We discuss this case in Section \ref{sec:Case1}. 

{\bf Case 2.} $f$ has translation length bounded above by some positive number $\la$. We discuss this case in Section \ref{sec:Case2}. 

\begin{remark}
1. For the constant $\la$ we will  take $\varepsilon(n,\kappa)/10$, where 
$\varepsilon(n, \kappa)$ is a positive lower bound for the Margulis constant of $X$. 

2. We need to use a power of $f$ only in Case 1, while in Case 2 we can take $N=1$. 
\end{remark}

We also note that if $f$ is  hyperbolic, the free group $\langle f^N, h \rangle$ constructed in our proof is convex-cocompact. See Proposition \ref{prop:gf-hyp} and Corollary \ref{schottky group}. 
One can also show that this subgroup is geometrically finite if $f$ is parabolic but we will not prove it. 

\medskip 
{\bf Acknowledgements.} The second author was partly supported by the NSF grant  DMS-16-04241 as well as by KIAS (the Korea Institute for Advanced Study) through the KIAS scholar program and by the Simons Foundation Fellowship, grant number 391602. The third author was partly supported by the NSF grant DMS-17-00814.

\section{Definitions and notation}

In a metric space $(Y, d)$, we will be using the notation $B(a,R)$ to denote the open $R$-ball  centered at $a\in Y$, and the notation $\bar{N}_R(A)$ to denote the closed $R$-neighborhood of a subset $A\subset X$. By
$$
d(A,B):= \inf \{d(a,b): a\in A, b\in B\}
$$
we denote the minimal distance between subsets $A, B\subset Y$. 

 \medskip
If $(Y,d)$ is a geodesic $\delta$-hyperbolic metric space or a $\mathrm{CAT}(0)$ space, then $\geo Y$ will denote the visual boundary equipped with the visual topology, and we write $\bar{Y}:= Y\cup \geo Y$. If $Y$ is proper then $\bar{Y}$ is a compactification of $Y$.  
Given a pair of points $x, y$ in $(Y,d)$, we will use the notation $xy$ to denote a geodesic segment 
in $Y$ connecting $x$ to $y$.  For general $\delta$-hyperbolic spaces this segment is \emph{not} unique, but, since 
any two such segments are within distance $\delta$ from each other, this abuse of notation is harmless. We  let $|xy|=d(x,y)$ denote the length of $xy$.  Given points $A, B, C\in Y$, we let $\triangle ABC$ denote a geodesic triangle in $Y$ with the vertices $A, B, C$. 
Similarly, if  $y\in Y, \xi\in \geo Y$, then  $y\xi$ will denote a geodesic ray emanating from $y$ and asymptotic to $\xi$. 

A subset $A$ of $Y$ is called $\la$-{\em quasiconvex} if every geodesic segment $xy$ with the end-points in $A$ is contained in $\bar{N}_\la(A)$. 

\medskip
A subset $A$ in a metric space $Y$ is called  \emph{starlike} with respect to a point $a\in A$ if for every $y\in A$ every geodesic segment $ya$ is contained in $A$. More generally, if $Y$ is $\delta$-hyperbolic or a $\mathrm{CAT}(0)$ space then $A\subset Y$ is 
called  \emph{starlike} with respect to a point $\xi\in \geo Y$ if for every $y\in A$ every geodesic ray $y\xi$ is contained in $A$. 

\medskip
Throughout the paper, $X$ will denote an $n$-dimensional  Hadamard manifold with sectional curvature ranging between $-\kappa^{2}$ and $-1$, unless otherwise stated. Let $d$ denote the Riemannian distance function on $X$. 
 We use $\geo X$ to denote the visual boundary of $X$, and $\bar{X}:= X\cup \geo X$ the visual compactification of $X$. Let $\Isom(X)$ denote the isometry group of $X$. We use $\varepsilon(n, \kappa)$ to denote a fixed positive lower bound on 
 the Margulis constant for $X$; this number is known to depend only on $n$ and $\kappa$, see e.g. \cite{BGS}.

Given a pair of points $p, q$ in $X$ we let $H(p,q)$ denote the closed half-space in $X$ given by
$$
H(p,q) = \{x\in X: d(x, p)\le d(x,q)\}. 
$$
Then $\Bis(p,q)=\Bis(q,p):= H(p,q)\cap H(q,p)$, is the equidistant set of $p, q$.  

We use the notation $\Hull(A)$ for the {closed convex hull} of a subset $A\subset X$ which is the intersection of all closed convex subsets of $X$ containing $A$.  

For each isometry $g$ of $X$ we define its {\em translation length} $\tau(g)$ as
\begin{equation}\label{eq:tr-length}
\tau(g) = \inf_{x\in X} d(x, g(x)).
\end{equation}
Isometries of $X$ are classified in terms of their translation lengths, see Section \ref{sec:cl-isom}.

A discrete  subgroup $\Gamma< \Isom(X)$ is called \emph{elementary} if  either it fixes a point in $\bar X$ or preserves a geodesic in $X$.

\section{Preliminary material}

\subsection{Some $\mathrm{CAT}(-1)$ computations}

Let $X$ be a $\mathrm{CAT}(-1)$ space. Recall that the hyperbolicity constant of $X$ is $\le \delta= \cosh^{-1}(\sqrt{2})$. 

\begin{lemma}\label{lem:RA-triangle}
Let $\triangle A_1A_2C$ be a triangle in  $X$ such that  $\angle A_1 C A_2\ge \pi/2$. 
Then 
$$
|A_1A_2| \ge  |A_1C| + |A_2C| -  2\delta. 
$$
\end{lemma}
\proof Let $D\in A_1A_2$ be the point closest to $C$. Then at least one of the angles $\angle A_iCD, i=1, 2$ is $\ge \pi/4$. 
The $\mathrm{CAT}(-1)$ property and the dual cosine law for the hyperbolic plane imply that
$$
\cosh(|CD|) \sin(\frac{\pi}{4})\le 1,
$$
i.e.
$$
|CD|\le \cosh^{-1}(\sqrt{2})=\delta. 
$$
The rest follows from the triangle inequalities. \qed

\begin{cor}\label{cor:2bisectors}
Suppose that $x, x_+, \hat{x}_+, x'_+$ are points  in $X$  which lie on a common geodesic and appear on 
this geodesic in the given order. Assume that 
$$
d(\hat{x}_+, x'_+)\ge d(x, x_+) + 2\cosh^{-1}(\sqrt{2}). 
$$
Then $H(x_+, \hat{x}_+)\subset H(x, x'_+)$. 
\end{cor}
\proof  We observe that the $\mathrm{CAT}(-1)$ condition implies that for each $z$ equidistant from $x_+, \hat{x}_+$ we have 
$$\angle z x_+ \hat{x}_+ \le \pi/2, \quad \angle z \hat{x}_+ {x}_+ \le \pi/2.$$
Hence,
$$
\angle x x_+ z\ge \pi/2, \quad \angle x'_+ \hat{x}_+ z\ge \pi/2.
$$
Then the lemma and  the triangle inequality implies that
$$
d(z, x)\le d(z, x'_+). 
$$
and, thus, 
$$
\Bis(x_+, \hat{x}_+)\subset H(x, x'_+).  
$$
Since every geodesic connecting  $w\in H(x_+, \hat{x}_+)$ to $x'_+$ passes through some point  $z\in \Bis(x_+, \hat{x}_+)$, it 
follows that 
\[
d(x, w)\le d(w, x'_+). \qedhere  
\]

\subsection{Quasiconvex and starlike subsets}

\begin{lemma}\label{qcstar}
Starlike subsets in a $\dl$-hyperbolic space $Y$ are $\dl$-quasiconvex.
\end{lemma}
\begin{proof} We prove this for subsets $A\subset Y$ starlike with respect to $a\in A$; the proof in the case of starlike subsets with respect to $\xi\in \geo Y$ is similar and is left to the reader. Take $z_1, z_2 \in A$. Then, by the $\delta$-hyperbolicity, 
\[
z_1z_2 \subset \nbhd{\dl}{az_1 \cup az_2} \subset \nbhd{\dl}{A}.\qedhere
\]
\end{proof}

Suppose now that $X$ is a Hadamard manifold of negatively pinched curvature as above. Then, according to   
 \cite[Proposition 2.5.4]{Bo2},  there exists  $\mathfrak{q}=\mathfrak{q}(\kappa,\la)$ such that 
 for every $\la$-quasiconvex subset $A\subset X$ we have 
\begin{equation}\label{eq:qcnbd}
\Hull(A)\subset \bar{N}_{\mathfrak{q}}(A).
\end{equation} 

In particular:

\begin{proposition}\label{prop:starhull}
For every  starlike subset  $A$ in a Hadamard manifold $X$ of negatively pinched curvature, the closed 
convex hull $\Hull(A)$ is contained in  the  $\mathfrak{q}=\mathfrak{q}(\kappa,\delta)$-neighborhood of $A$.
\end{proposition}

In what follows, we will suppress the dependence of ${\mathfrak q}$ on $\kappa$ and $\delta$ since these are fixed for our space $X$.

\subsection{Classification of isometries}\label{sec:cl-isom}

Let $X$ be a negatively curved Hadamard manifold. Isometries of $X$ are classified into three types according to their translation lengths 
$\tau$, see \cite{BGS, BCG}. 

\medskip
1. An isometry $g$ of $X$ is \emph{hyperbolic} if $\tau(g)>0$.  Equivalently,  
the infimum in \eqref{eq:tr-length} is attained and is positive. In this case, the infimum 
 is attained on a $g$-invariant geodesic, called the {\em axis} of $g$, and denoted by $A_g$. 

\medskip 
2. An isometry $g$ of $X$ is {\em elliptic} of $\tau(g)=0$ and the infimum in \eqref{eq:tr-length} is attained; the set where the infimum is attained is a totally geodesic submanifold of $X$ fixed pointwise by $g$. 

\medskip
3. An isometry $g$ of $X$ is {\em parabolic} if the infimum in \eqref{eq:tr-length} is not attained.  In this case, the infimum is necessarily 
equal to zero. 

\medskip
Thus, only parabolic and elliptic isometries have zero translation lengths. For any $g\in\Isom(X)$ and $m\in \Z$ we have
\begin{equation}\label{eq:tl}
\tau(g^m)=|m| \tau(g). 
\end{equation}

\medskip
The following theorem provides an alternative characterization of types of isometries of $X$, see \cite{CDP}.

\begin{thm}\label{thm:classification-of-isometries}
Suppose that $g$ is an isometry of $X$. Then:
\begin{enumerate}
\item $g$ is hyperbolic if and only if for some (equivalently, every) $x\in X$ the orbit map $N\to g^Nx$ is a quasiisometric embedding $\Z\to X$. 

\item $g$ is elliptic if and only if  for some (equivalently, every) $x\in X$ the orbit map $N\to g^Nx, N\in \Z$ has bounded image.  

\item $g$ is parabolic if and only if for some (equivalently, every) $x\in X$ the orbit map $N\to g^Nx, N\in \Z$ is  proper and
$$
\lim_{N\to\infty} \frac{d(x, g^N(x))}{N}=0. 
$$
\end{enumerate}
\end{thm}

If $f, g$ are hyperbolic isometries of $X$ generating a discrete subgroup of $\Isom(X)$, then either the ideal boundaries of the axes $A_f, A_g$ are disjoint or $A_f=A_g$, (see \cite{Bo1}, the argument for negatively curved Hadamard manifolds is similar).

\subsection{Margulis cusps and tubes}

Take $g\in \Isom(X)$.  For each $\varepsilon\ge \tau(g)$ we define the following nonempty closed convex subset of $X$: 
$$
T_{\varepsilon}(g)=\lbrace x\in X\mid d(x ,g(x))\leq \varepsilon \rbrace.
$$
Of primary importance are  subsets $T_{\varepsilon}(g)$ for $\varepsilon < \varepsilon(n, \kappa)$.  For any two isometries $g, h$ 
of $X$ we have
\begin{equation} \label{eq3.1}
T_{\varepsilon}(hgh^{-1})= h( T_{\varepsilon}(g)).
\end{equation}
In particular, if $g, h$ commute, then $h$ preserves $T_{\varepsilon}(g)$.

For parabolic isometries $g$ of $X$ define the {\em Margulis cusp}
$$
\mathcal{T}_{\varepsilon}(g):= \bigcup_{i\in {\mathbb Z} - \{0\}} T_{\varepsilon}(g^i). 
$$
(The same definition works for elliptic isometries of $X$, except the above region is not called a {\em cusp}.) This subset is $\langle g\rangle$-invariant.

Suppose that $g$ is a hyperbolic isometry of $X$. Define  $m_{g}$ to be the (unique) 
positive integer such that 
\begin{equation}\label{eq:mg}
\tau(g^{m_{g}})\leq \varepsilon/10, \quad
\tau(g^{m_{g}+1})> \varepsilon/10,
\end{equation}   
and set 
$$
\mathcal{T}_{\varepsilon}(g): = \bigcup_{1\le i \leq m_g} T_{\varepsilon}(g^{i})\subset X.
$$
If $\tau(g)>\varepsilon/10$, then $\mathcal{T}_{\varepsilon}(g)=\emptyset$. 

Since the subgroup $\langle g\rangle$  is abelian, in view of \eqref{eq3.1}, we obtain: 

\begin{lemma}\label{lem3.1}
The subgroup $\langle g\rangle$ preserves $\mathcal{T}_{\varepsilon}(g)$ and, hence,  also 
 preserves $\Hull(\mathcal{T}_{\varepsilon}(g))$. 
\end{lemma}

By convexity of the distance function, for any isometry $g\in\Isom(X)$,  $T_{\varepsilon}(g)$ is convex. In particular, $\TT(g)$ is a starlike region with respect to any fixed point $p\in \bar{X}$ of $g$ for general $g$, and with respect to 
 and any point on the axis of $g$ if $g$ is hyperbolic.
As a corollary to Lemma \ref{qcstar}, one obtains,

\begin{corollary}\label{qctube}
For every isometry $g\in\Isom(X)$, the set  $\TT(g)$ is $\dl$-quasiconvex.
\end{corollary}

 Proposition \ref{prop:starhull} then implies 

\begin{corollary}\label{cor:hull}
For every isometry $g\in \Isom(X)$,
\[
\Hull(\TT(g)) \subset \nbhd{\mathfrak{q}}{\TT(g)},
\]
where $\mathfrak{q}$ is as  in Proposition \ref{prop:starhull}.
\end{corollary}

For a more detailed discussion of the regions $\TT(g)$, see \cite{ Bo2, KL1}.

\subsection{Displacement estimates}
In this subsection, we let $X$ be a $\mathrm{CAT}(-1)$ geodesic metric space.
For each pair of points $A, B\in \H^2$ and each circle $S\subset \H^2$ passing through these points, we let $\arc{AB}^S$ denote 
the (hyperbolic) length of the shorter arc into which $A, B$ divide the circle $S$.

\begin{lemma}\label{lem:L1}
If $d(A,B)\le D$ then for every circle $S$ as above, then the length $\ell$ of $\arc{AB}^{S}$ satisfies the inequality: 
$$
d(A,B) \le \ell\le \frac{2\pi \tanh(D/4)}{1- \tanh^2(D/4)}.  
$$
\end{lemma}

\begin{proof}
 The first inequality is clear, so we verify the second. We want to maximize the length of $\arc{AB}^S$ among all circles $S$ passing through $A, B$. We claim that the maximum is achieved on the circle $S_o$ whose center $o$ is the midpoint of $AB$.  This follows from the fact that given any other circle $S$, we have the radial projection from $\arc{AB}^{S_o}$ to $\arc{AB}^S$ (with the center of the projection 
at $o$). Since this radial projection is distance-decreasing (by convexity), the claim follows.  The rest of the proof  amounts to 
a computation of the length of the hyperbolic half-circle with the given diameter. 
\end{proof}

\begin{lemma}
\label{decrease speed}
There exists a function $c(D)$ so that the following holds.  Consider
an isosceles triangle $ABC$ in $X$ with $d(A,C)=d(B,C)$, $d(A,B)\le D$, and an isosceles subtriangle $A'B'C$ with $A'\in AC, B'\in BC$, $d(A,A')=d(B,B')=\tau$.  Then 
$$
d(A', B')\leq c (D) e^{-\tau}. 
$$
\end{lemma}
\begin{proof} In view of the $\mathrm{CAT}(-1)$  assumption, it suffices to consider the case when $X=\H^2$. We will work with the unit disk model of 
the hyperbolic plane where $C$ is the center of the disk as in Figure \ref{disk triangle}.
\begin{figure}[h]
\centering
\includegraphics[width=2.0in]{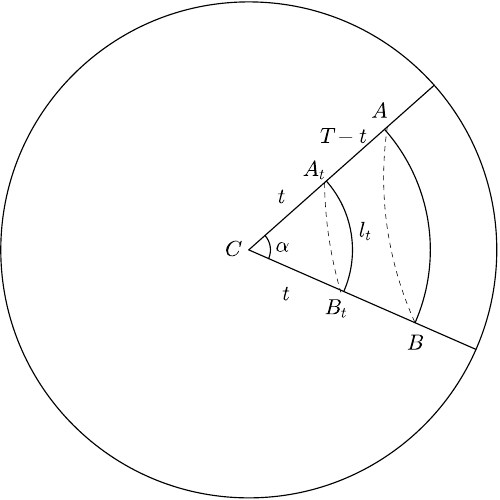}
\caption{ \label{disk triangle}}
\end{figure}
  Let $\al$ denote the angle $\angle(ACB)$.  Set $T:=d(C,A)=d(C,B)$. 
For points $A_t\in CA, B_t\in CB$ such that $d(C,A_t)=d(C,B_t)=t$ we let $l_t$ denote the hyperbolic length of the (shorter) 
circular arc  $\arc{A_{t}B_{t}}=\arc{A_{t}B_{t}}^{S_t}$ of the angular measure $\al$, centered at $C$ and connecting $A_t$ to $B_t$.  
(Here $S_t$ is the  circle centered at $C$ and of the hyperbolic radius $t$.) 
Let $R_t$ denote the Euclidean distance between $C$ and $A_t$ (same for $B_t$).  
Then 
$$
l_{t}=\dfrac{2\alpha R_{t}}{1-R_{t}^{2}},
$$
$$
R_{t}=\tanh(t/2). 
$$
Thus,  for $\tau=T-t$, 
\begin{multline*}
\dfrac{l_{t}}{l_{T}}=\dfrac{R_{t}}{R_{T}}\dfrac{1-R_{T}^{2}}{1-R_{t}^{2}}\leq 
\dfrac{1-R_{T}^{2}}{1-R_{t}^{2}}\leq 
2\dfrac{1-R_{T}}{1-R_{t}}\\
=2\dfrac{1-\tanh(T/2)}{1-\tanh(t/2)}=2\dfrac{e^{t}+1}{e^{T}+1}
=2\dfrac{e^{-T}+e^{-\tau}}{e^{-T}+1}\leq 4e^{-\tau}.
\end{multline*}
 In other words, 
$$
d(A_t,B_t)\le l_t \leq 4e^{-\tau}l_{T}. 
$$
Combining this inequality with Lemma \ref{lem:L1}, we obtain 
$$
l_t\le 4e^{-\tau} \frac{2\pi \tanh(d(A,B)/4)}{1- \tanh^2(d(A,B)/4)}\le 4e^{-\tau} \frac{2\pi \tanh(D/4)}{1- \tanh^2(D/4)}. 
$$
Lastly, setting $A'=A_t, B'=B_t, A=A_T, B=B_T$, we get: 
\[
d(A',B')\le  4 \frac{2\pi \tanh(D/4)}{1- \tanh^2(D/4)} e^{-\tau}= c(D) e^{-\tau}. \qedhere
\]
\end{proof}

\begin{corollary}
\label{margulis distance}
There exists a function $\mathfrak{r}(\varepsilon)$ such that for any hyperbolic isometry $h\in \Isom(X)$ with translation length 
$$
\tau(h)=l \leq  \varepsilon/10,$$ 
if $A\in X$ satisfies $d(A, h(A))=\varepsilon$, then there exists $B\in X$ such that $d(B, h(B))=\varepsilon/3$, $d(A, B)\leq \mathfrak{r}=\mathfrak{r}(\varepsilon)$ and $B$ lies 
of the shortest geodesic segment connecting  $A$ to the axis $A_{h}$ of $h$.  
\end{corollary}

\begin{proof}
Let $C\in A_{h}$ be the closest point to $A$ in $A_{h}$. By the convexity of the distance function, there exists a point $B\in AC$ such that $d(B, h(B))=\varepsilon/3$. Suppose that $d(A, B)=d(h(A), h(B))=t$ and $d(A, C)=d(h(A), h(C)))=T$ as shown in Figure \ref{quidrilateral split}. Then $d(C, h(A))\leq T+l \leq T+\varepsilon/10$. There exist points $D, E$ in the segment $h(A) C$ such that $d(C, D)=d(C, B)=T-t$,  
$d(h(A), E)=t$ and $d(A', C)=d(A, C)=T$.

Then $d(A, A')\leq \varepsilon+l \leq 11\varepsilon/10$. 
By Lemma \ref{decrease speed},   $c(11 \varepsilon/10)$ (defined in that lemma) satisfies 
$$
d(B, D)\leq  c(d(A, A')) e^{-t}\le c(11\varepsilon/10) e^{-t}. 
$$ 
Similarly, by taking the point  $A''\in h(A)C$ satisfying $d(A'', h(A))=T$,  $d(h(C), A'')\le 2l$, 
considering the isosceles triangle $\triangle h(C)A'' h(A)$ and its subtriangle $\triangle h(B) E h(A)$,  we obtain:
$$
d(h(B), E)\leq  c(2l) e^{t-T}. 
$$

Since $l\leq \varepsilon/10$ and $d(B, h(B))=\varepsilon/3$,  convexity  
of the distance function implies that  $T-t> t$. Thus, 
\begin{multline*}
\varepsilon/3 =d(B, h(B))\leq d(B, D)+ d(D, E)+ d(E, h(B)) \\
\leq c(11\varepsilon/10) e^{-t} + l +  c(2l) e^{t-T}\le
 c(11\varepsilon/10) e^{-t} + \frac{\varepsilon}{10} + 
 c({\varepsilon/5}) e^{-t},
 \end{multline*}
 which simplifies to
 $$
\frac{7}{30}\varepsilon \le  \left( c(11\varepsilon/10)  + c(\varepsilon/5) \right) e^{-t}, 
$$
and consequently
\[
d(A,B)= t\le \mathfrak{r}(\varepsilon):=  \log \left (  [c(11\varepsilon/10)  + c(\varepsilon/5)]  \frac{30}{7}\varepsilon^{-1}  \right).  \qedhere
\]
\end{proof}

\begin{figure}[H]
\centering
\includegraphics[width=2.5in]{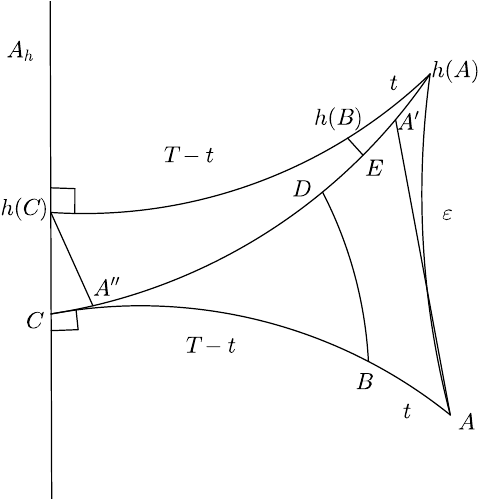}
\caption{ \label{quidrilateral split}}
\end{figure}

\subsection{Local-to-global principle for quasigeodesics in $X$}

For a piecewise-geodesic path consisting of alternating `long'  arcs and `short' segments  such that adjacent geodesic segments meet at the angles $\geq \pi/2$, we construct a quasigeodesic in $X$ by making the long segments sufficiently long, given   a lower bound on the lengths of the short arcs. More precisely, according to \cite[Proposition 7.3]{KL1}:

\begin{proposition}
\label{qua}
There are  functions   $\lambda=\la(\varepsilon)\ge 1, \alpha=\al(\varepsilon)\ge 0$ and  $L=L(\varepsilon)>\varepsilon>0$ 
such that the following  holds. Suppose that $\gamma=\cdots \gamma_{-1}\ast \gamma_{0}\ast \gamma_{1}\ast \cdots \ast \gamma_{n}... \subseteq X$ is a piecewise geodesic path  such that:
\begin{enumerate}
\item Each geodesic arc $\gamma_{j}$ has length either at least $\varepsilon$ or at least $L$. 

\item 
If $\gamma_{j}$ has length $<L$, then the adjacent geodesic arcs $\gamma_{j-1}$ and $\gamma_{j+1}$ have lengths at least $L$. 

\item All adjacent geodesic segments  meet at the angles $\ge \pi/2$. 
\end{enumerate}
Then $\gamma$ is a $(\la,\alpha)$-quasigeodesic in $X$.  
\end{proposition}

\subsection{Ping-pong}

\begin{proposition}
\label{free group 1}
Suppose that $g, h\in \Isom(X)$ are parabolic/hyperbolic  elements with equal translation lengths $\leq \varepsilon/10$, and 
\begin{equation}
d(\Hull(\TT(g)), \Hull(\TT(f)))\ge L,
\end{equation}
where $L=L(\varepsilon/10)$ is as in Proposition \ref{qua}. 
Then $\Phi:=\langle g, h\rangle < \Isom(X)$ is a free subgroup of rank 2.  
\end{proposition}

\begin{proof}
To simplify notation, for a non-elliptic element $f\in\Isom(X)$, we denote $\Hull(\TT(f))$ by
$\TH(f)$.

Using Lemma \ref{lem3.1}, (\ref{eq3.1}), and the definition of $\TT$, we obtain
$$d(\TH(g), g^{k} \TH(h))=d(\TH(g), \TH(h))\ge L, k\in \Z.$$

Our goal is to show that  every  nonempty word $w(g,h)$ represents a nontrivial element of $\Isom(X)$. It suffices to consider cyclically 
reduced words $w$ which are not powers of $g, h$. 

\medskip
We will consider a cyclically reduced word  
\begin{equation}\label{eq:w}
 w =w(g,h) = g^{m_k}h^{m_{k-1}}g^{m_{k-2}}h^{m_{k-3}} \dots g^{m_{2}}h^{m_1}, 
 \end{equation}
words with the last letter $g$ are treated by relabeling.  Since $w$ is cyclically reduced and is not a power of $g, h$, 
the number $k$ is $\ge 2$ and  all of the $m_i$'s in this equation are nonzero.

For each $N\ge 1$ we define the {\em $r$-suffix} of $w^N$ as the following subword of $w^N$:  
\begin{equation}\label{eq:wr}
w_r = \left\{
\begin{array}{rl}
g^{m_r}h^{m_{r-1}}g^{m_{r-2}}h^{m_{r-3}} \dots g^{m_{2}}h^{m_1}, & r\text{ even}\\
h^{m_{r}}g^{m_{r-1}}h^{m_{r-2}} \dots g^{m_{2}}h^{m_1}, & r\text{ odd}\\
\end{array}\right.
\end{equation}
where, of course, $m_i\equiv m_j$ modulo $N$.  Since $w$ is  reduced,   each $w_r$ is reduced as well.

We will prove that the map
\[\Z \rightarrow X, \quad N\mapsto w^N x\]
is a quasiisometric embedding. This will imply that $w(g,h)$ is nontrivial. In fact, this will also show that $w(g,h)$ is hyperbolic, see Theorem 
\ref{thm:classification-of-isometries}.

\medskip 
Let $l=yz$ be the unique shortest  geodesic segment connecting points in $\TH(g)$ and $\TH(h)$,  where 
 $y\in \TH(g)$ and $z\in \TH(h)$. For $r\ge 0$, we denote $w_r l$, $w_{r} y$ and $w_{r}z$ 
by $l_r$, $y_r$ and $z_r$, respectively.   In particular, $y_0=y, z_0=z$ and $l_0=l$.

Since $l$ is the shortest segment between  $\TH(g), \TH(h)$  and these are convex subsets of $X$, for every 
$y' \in \TH(g)$ (resp. $z'\in \TH(h))$,
\begin{equation}\label{eq:angle}
\angle y'yz \ge \pi/2, \quad 
(\text{resp. } \angle yzz' \ge \pi/2).
\end{equation}

Since $g$ and $h$ have equal translation lengths, $h$ is parabolic (resp. hyperbolic) if and only if $g$ is parabolic (resp. hyperbolic). 
When both of them are hyperbolic, since $y$ and $z$ are not in the interior of $\TT(g)$ and $\TT(h)$, respectively, $d(y, g^i y),d(z,h^j z) \ge \ve$, for all $1\le i\le m_g$, $1\le j\le m_h$. Also, when $i> m_g$, $j>m_h$, it follows from (\ref{eq:tl}) and (\ref{eq:mg}) that
\[
\min\left( d(y, g^i y),d(z,h^j z)\right) \ge \frac{\ve}{10}.
\]
Moreover, when both $g$ and $h$ are parabolic, $d(y, g^i y),d(z,h^j z) \ge \ve$, for all $1\le i$, $1\le j$.
Therefore, in the general case,
\begin{equation}\label{eq:tl2}
\min\left( d(y, g^i y),d(z,h^j z)\right)  \ge \frac{\ve}{10}, \quad \forall i\ge1, \forall j\ge 1.
\end{equation}

\begin{figure}[H]
\centering
\includegraphics[width=5in]{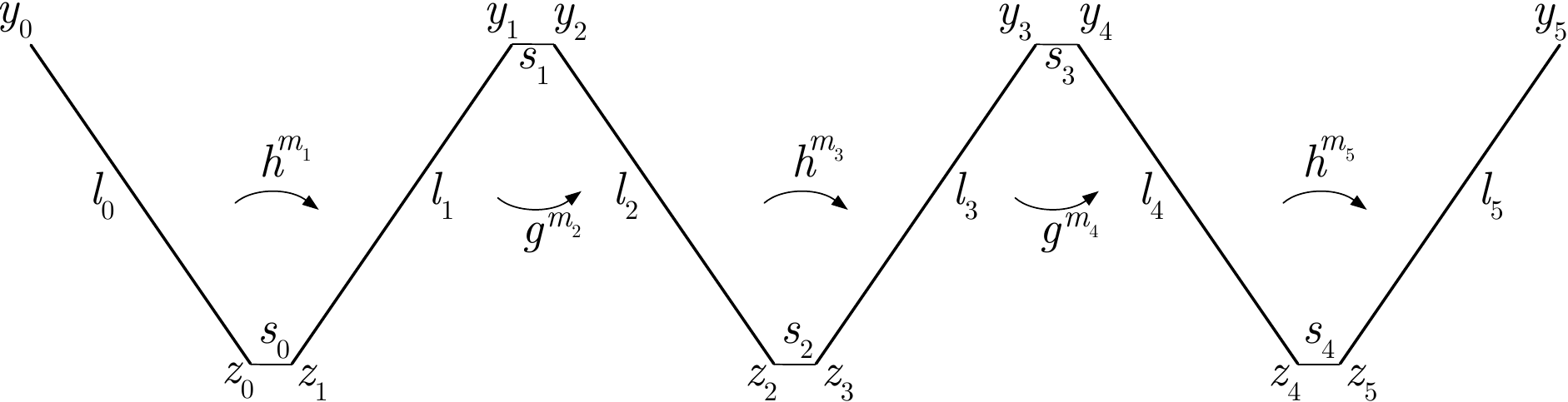}
\caption{}\label{fig:3.1}
\end{figure}

Let $s_r$ be the segment 
\[
s_r = \left\{
\begin{array}{rl}
y_ry_{r+1}, & \text{when }r\text{ is odd}\\
z_rz_{r+1}, & \text{when }r\text{ is even}\\
\end{array}\right..
\]
See the arrangement of the points and segments in Figure \ref{fig:3.1}.

Let $\tilde{l}_N$ be the concatenation of the segments $l_r$'s and $s_r$'s as shown in Figure \ref{fig:3.1}, $0\le r\le kN$. 
According to  (\ref{eq:tl2}), the length of each segment $s_r$ is at least $\ve/10$, while by the assumption, the length of each $l_r$ is $\ge L=L(\varepsilon/10)$. Moreover, according to (\ref{eq:angle}), the 
angle between any two consecutive segments in $\tilde{l}_N$ is at least $\pi/2$. 
Using Proposition \ref{qua}, we conclude that $\tilde{l}_N$ is a $(\lambda,\alpha)$-quasigeodesic.

Consequently,
\begin{equation}
\label{eq:qiw}
d(w^{N}x, x) \ge \frac{1}{\lambda} 
\left( \sum_{i=0}^{kN-1} |s_i| + NkL\right)  - \alpha 
\ge \frac{kL}{\lambda} N  - \alpha. 
\end{equation}
From this inequality it follows that the map $\Z\to X, N\mapsto w^Nx$ is a quasiisometric embedding.  
\end{proof}

\begin{remark}
In fact, this proof also shows that every nontrivial element of the subgroup $\Phi< \Isom(X)$ is either conjugate to one of the generators or is 
hyperbolic.  
\end{remark}

For the next proposition and the subsequent remark, 
one needs the notions of {\em convex-cocompact} and {\em geometrically finite} subgroups of $\Isom(X)$. We refer  to \cite{Bo2} for several equivalent definitions, see also \cite[section 1]{Kapovich-Leeb}. For now, it suffices to say that a subgroup $\Ga$ in $\Isom(X)$ is {\em convex-cocompact} if it is finitely generated and for some (equivalently, every) $x\in X$, the orbit map $\Ga \to \Ga x\subset X$ is a quasiisometric embedding, where $\Ga$ is equipped with a word metric.

\begin{proposition}\label{prop:gf-hyp}
Let $g,h\in \Isom(X)$ be hyperbolic isometries satisfying the hypothesis of Proposition \ref{free group 1}. Then the subgroup $\Phi=\langle g,h\rangle< \Isom(X)$ is convex-cocompact.
\end{proposition}
\begin{proof} We equip the free group $\F_2$ on two generators  (denoted $g, h$) with the word metric 
corresponding to this free generating set. Since $g,h$ are hyperbolic, by (\ref{eq:tl}) the lengths of the segments $s_r$'s in the proof of 
Proposition \ref{free group 1} are $\ge \tau |m_{r+1}|$, where
\[\tau = \tau(g)=\tau(h).\] Then, for $N=1$, $r=k$, and a reduced but not necessarily cyclically reduced word $w$, the inequality (\ref{eq:qiw}) becomes 
\begin{equation}
d(wy,y)\ge   \frac{1}{\lambda} 
\left( \sum_{i=0}^{k-1} |s_i|\right)  - \alpha \ge  \frac{\tau}{\lambda}  |w| -\alpha,
\end{equation}
where $|w| \ge  |m_1| + |m_2| + \dots + |m_{k}|$ is the (word) length of $w$.  
 Therefore, the orbit map $\F_2  \rightarrow \Phi y \subset X$ is a quasiisometric embedding.
\end{proof}

\begin{rem}
One can also show that if $g, h$ are parabolic then the subgroup $\Phi$ is geometrically finite. We will not prove it in this paper since a proof requires further geometric background material on geometrically finite groups. 
\end{rem}

\section{Case 1: Displacement bounded below} \label{sec:Case1} 

In this section we consider discrete nonelementary subgroups of $\Isom(X)$ generated by two hyperbolic elements $g, h$ whose translation lengths are equal to $\tau\ge \la$.  Our goal is to show that in this case the subgroup $\langle g^N, h^N\rangle$ is free of rank 2 provided that $N$ is greater than some constant depending only on the Margulis constant of $X$ and on $\la$. 
The strategy is to bound from above how `long'  
 the axes $A_g, A_h$ of $g$ and $h$ can stay `close to each other' in terms of the constant $\la$. Once we get such an estimate, we find a uniform upper bound on $N$ such that Dirichlet domains for $\langle g^N\rangle, \langle h^N\rangle$ (based at some points on $A_g, A_h$)   have disjoint complements.  This implies that $g^N, h^N$ generate a free subgroup of rank two by a classical ping-pong argument.  

\medskip
Let $\alpha, \beta$ be complete geodesics in the Hadamard manifold $X$. These geodesics eventually will be the axes of $g$ and $h$, hence, 
we assume that these geodesics do not share ideal end-points. Let $x_- x_+$ denote the (nearest point) projection of $\beta$ to 
$\alpha$ and let $y_-y_+$ denote the projection of  $x_- x_+$ to $\beta$.  
Let $x, y$ denote the mid-points of $x_-x_+$ and $y_-y_+$ respectively. 

Then 
$$
L_\beta:= d(y_-, y_+)\le L_\al:=d(x_-, x_+). 
$$

Fix some $T\ge 0$, and let $\hat{x}_- \hat{x}_+$ and 
$\hat{y}_- \hat{y}_+$ denote the subsegments of $\al$ and $\beta$ containing $x_- x_+$  and  $y_-y_+$ respectively, such that 
\begin{equation}\label{eq:T}
d(x_\pm, \hat{x}_\pm)=T,\quad  d(y_\pm, \hat{y}_\pm)=T. 
\end{equation}
We let $U_\pm$ and $V_\pm$ denote the `half-spaces' in $X$ equal to $H(\hat{x}_\pm, x_\pm)$ and 
$H(\hat{y}_\pm, y_\pm)$ respectively. See Figure \ref{figure2.fig}. 

\begin{figure}[h]
\begin{center}
\begin{tikzpicture}
    \node[anchor=south west,inner sep=0] (image) at (0,0,0) {\includegraphics[width=3in]{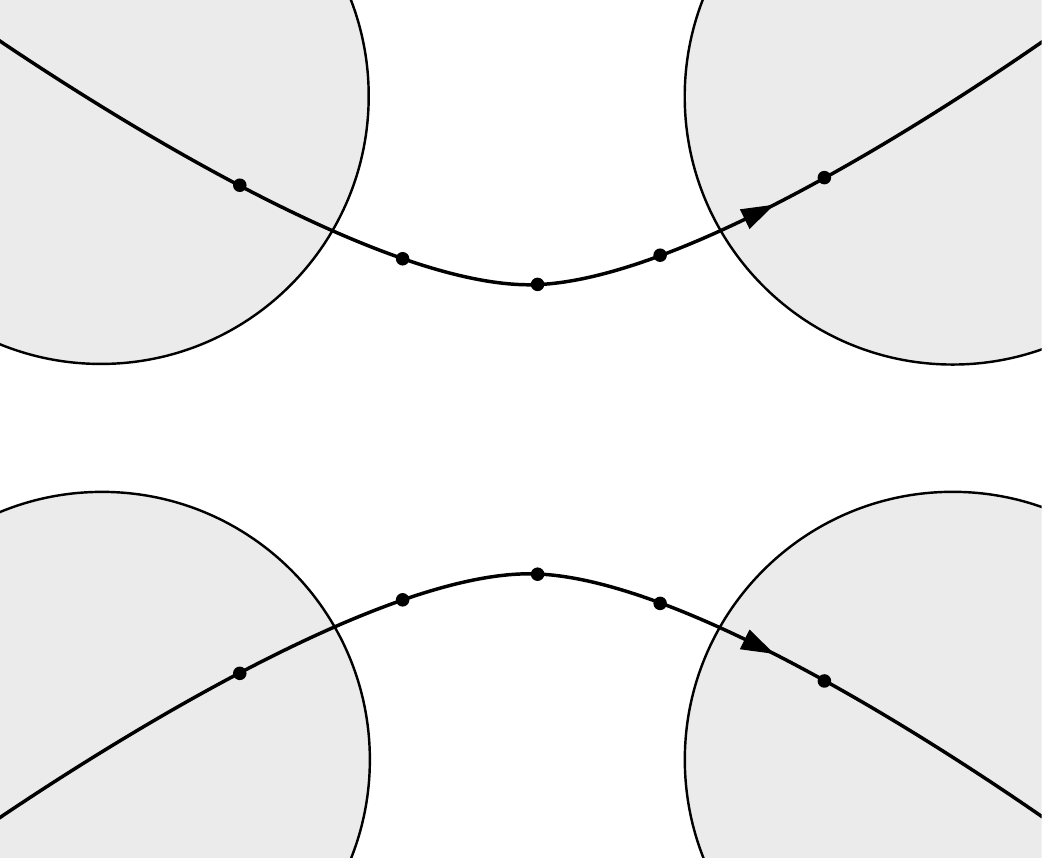}};
    \begin{scope}[x={(image.south east)},y={(image.north west)}]
        \node at (.52,.71) {$x$};
        \node at (.395,.73) {$x_-$};
        \node at (.635,.74) {$x_+$};
        \node at (.795,.838) {$\hat{x}_+$};
        \node at (.235,.825) {$\hat{x}_-$};
        \node at (.05,.7) {$U^-$};
        \node at (.95,.7) {$U^+$};
        \node at (.73,.71) {$g$};
        \node at (.03,.97) {$\alpha$};
        \node at (.52,.28) {$y$};
        \node at (.395,.26) {$y_-$};
        \node at (.635,.26) {$y_+$};
        \node at (.795,.16) {$\hat{y}_+$};
        \node at (.235,.15) {$\hat{y}_-$};
        \node at (.05,.3) {$V^-$};
        \node at (.95,.3) {$V^+$};
        \node at (.73,.29) {$h$};
        \node at (.03,.12) {$\beta$};
    \end{scope}
\end{tikzpicture}
\end{center}
\caption{}
\label{figure2.fig}
\end{figure}

The following is proven in \cite[Appendix]{BCG}:

\begin{lemma}\label{lem:BCG}
If $T\ge 5$ then the sets $U_\pm$, $V_\pm$ are pairwise disjoint. 
\end{lemma}

Suppose now that $g, h$ are hyperbolic isometries of $X$ with the axes $\al, \beta$ respectively, and equal translation length $\tau(g)=\tau(h)=\tau \ge \la>0$.  We let $\Ga=\langle g, h\rangle < \Isom(X)$ denote the,   necessarily nonelementary (but not necessarily discrete),  
 subgroup of isometries of $X$ generated by $g$ and $h$.

As an application of the above lemma, as in \cite[Appendix]{BCG}, we obtain:

\begin{lemma}
If $N\tau\ge L_\alpha + 5 + 2\delta$
then the half-spaces $H(g^{\pm N}x, x)$, $H(h^{\pm N}y, y)$ are pairwise disjoint.
\end{lemma}
\proof The inequality  
$$
N\tau\ge L_\alpha + 5 + 2\delta\ge L_\beta + 5 + 2\delta.
$$ 
implies that the quadruples 
$$
(x, x_+, \hat{x}_+,  g^N(x)), (x, x_-, \hat{x}_-,  g^{-N}(x)), (y, y_+, \hat{y}_+,  h^N(y)), (y, y_-, \hat{y}_-,  h^{-N}(y))
$$
satisfy the assumptions of Corollary \ref{cor:2bisectors} where $\hat{x}_\pm$ and $\hat{y}_\pm$ are given by taking $T=5$ in (\ref{eq:T}). Therefore, according to this corollary, we have
$$
H(g^{\pm N}(x), x)\subset U^\pm, \quad H(h^{\pm N}(y), y)\subset V^\pm. 
$$
Now, the assertion of the lemma follows from Lemma \ref{lem:BCG}.  \qed

\begin{cor}\label{cor:N-in}
If 
\begin{equation}\label{eq:N-in}
N\tau \ge L_\alpha + 5 + 2\delta
\end{equation}
 then the subgroup $\Ga_N< \Ga$ generated by $g^N, h^N$ is free with the basis $g^N, h^N$.  
\end{cor}
\proof We have 
$$
g^{\pm N} \left( H(h^{-N}(y), y)   ~\cup~ H(h^{+N}(y), y ) \right)\subset H(g^{\pm N}x, x)
$$
and 
$$
h^{\pm N} \left( H(g^{-N}(x), x)   ~\cup~ H(g^{+N}(x), x) \right) \subset H(h^{\pm N}y, y). 
$$
Thus, the conditions of the standard ping-pong lemma (see e.g. \cite{Har, DK}) are satisfied and, hence,  
$\Ga_N$ is free with the basis $g^N, h^N$. \qed 

\medskip
Let $\eta=d(\al,\beta)$ denote the minimal distance between $\al, \beta$ and pick some $\eta_0>0$ (we will eventually take $\eta_0= 0.01 \varepsilon(n,\kappa)$).  
Let $\beta_0= z_-^0 z_+^0\subset \beta$ be the (possibly empty!) maximal closed subinterval  such that the distance from the end-points of $\beta_0$ to $\al$ is $\le \eta_0$.  Thus, $\beta_0 \subset \bar{N}_{\eta_0}(\al)$. 

\begin{rem}
$\beta_0=\emptyset$  if and only if $\eta_0< \eta$. 
\end{rem}

\begin{figure}[h]
\begin{center}
\begin{tikzpicture}
    \node[anchor=south west,inner sep=0] (image) at (0,0,0) {\includegraphics[width=3in]{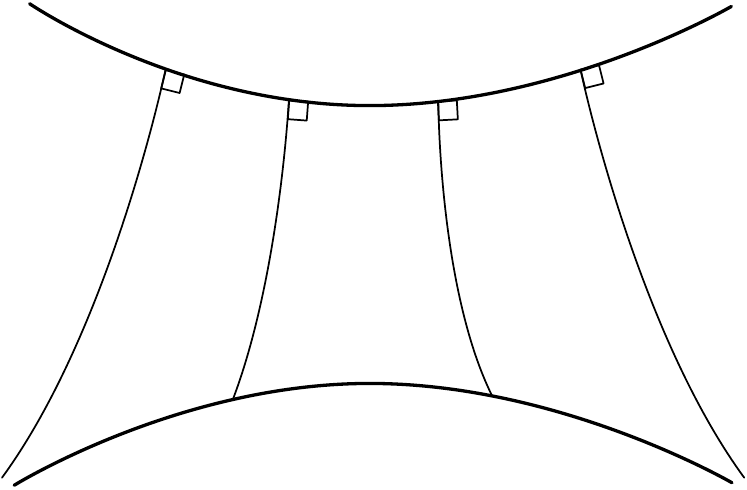}};
    \begin{scope}[x={(image.south east)},y={(image.north west)}]
        \node at (.24,.89) {$x_-$};
	\node at (.78,.905) {$x_+$};
	\node at (.5,.83) {$2L_0$};
	\node at (.3,.78) {$L_1$};
	\node at (.7,.77) {$L_1$};
	\node at (.33,.5) {$\eta_0$};
	\node at (.64,.5) {$\eta_0$};
	\node at (.5,.16) {$\beta_0$};
	\node at (.95,1) {$\alpha$};
	\node at (.93,.0) {$\beta$};
    \end{scope}
\end{tikzpicture}
\end{center}
\caption{}
\label{figure1.fig}
\end{figure}

Let $\al_0= x_-^0 x_+^0$ denote the projection of $\beta_0$ to $\al$, let $2L_0$ denote the length of $\al_0$.  Hence, the intervals $\al_0, \beta_0$ are within Hausdorff distance $\eta_0$ from each other.  

Furthermore, 
$\angle \beta(-\infty) z_-^0 x_-^0\ge \pi/2$ and $\angle \beta(-\infty) z_+^0 x_+^0\ge \pi/2$; 
see Figure \ref{figure1.fig}.  Hence, according to \cite[Corollary 3.7]{KL1}, for 
$$
L_1=\sinh^{-1}\left(  \frac{1}{\sinh(\eta_0)} \right),
$$
we have
$$
d(x_-, x_-^0)\le L_1, \quad d(x_+, x_+^0)\le L_1. 
$$
Thus, the interval $x_-x_+$ breaks into the union of two subintervals of length $\le L_1=L_1(\eta_0)$ and the interval $\al_0$ of the length $2L_0$.  In other words,  $L_\al=2(L_0+L_1)$.   

Most of our discussion below deals with the case when the interval $\beta_0$ is nonempty.

\medskip 
Our goal is to bound from above $L_\al$ in terms of $\la, \eta_0$ and the Margulis constant $\varepsilon(n, \kappa)$ of $X$, provided that 
$\eta_0=0.01\varepsilon(n, \kappa)$ and $\Ga$ is discrete.

\begin{lemma}
Let $S\subset \Ga$ be the subset consisting of elements of word-length $\le 4$ with respect to the generating set $g, h$. 
Let $P_-P_+\subset \al_0$ be the middle  subinterval of $\al_0$ whose length is 
$\frac{2}{9}L_0$. Assume that $\tau\le d(P_-,P_+)$. Then for each $\gamma\in S$ the interval 
$\gamma(P_-P_+)$ is contained in the $3\eta_0$-neighborhood of $\al_0$. 
\end{lemma}
\proof The proof is a straightforward application of the triangle inequalities taking into account the fact that the Hausdorff distance between 
$\al_0$ and $\beta_0$ is $\le \eta_0$.  \qed

\medskip 
Then, arguing as in the proof of \cite[Theorem 10.24]{Kapovich-book}\footnote{In fact, the argument there is a variation on a proof due to Culler-Shalen-Morgan and Bestvina, Paulin}, 
we obtain that each of the commutators
$$
[g^{\pm 1}, h^{\pm 1}], \quad [h^{\pm 1}, g^{\pm 1}]
$$
moves each point of $P_-P_+$ by at most 
$$28\times 3\eta_0\le 100 \eta_0.$$ 
Therefore, by applying the Margulis Lemma as in the  proof of \cite[Theorem 10.24]{Kapovich-book}, we obtain:

\begin{cor}
If  $\Ga$ is discrete and $\eta_0= 0.01 \varepsilon(n,\kappa)$, then 
$$
\tau\ge \frac{2}{9}L_0= \frac{1}{9}(L_\al - 2L_1). 
$$ 
\end{cor}

\begin{cor}\label{cor:dychotomy}
If $\Ga$ is discrete and $\tau\ge \la$, then  the subgroup 
$\langle g^N, h^N\rangle =\Ga_N <\Ga$ is free of rank 2 whenever one of the following holds:  

i. Either $L_\al\le 3L_1$ and 
$$
N\ge  \frac{5+ 2\delta + 3L_1}{\la}. 
$$

ii. Or $L_\al \ge 3L_1$ and 
$$
N\ge  27 + \frac{9(5+2\delta)}{L_{1}}.   
$$
\end{cor}
\proof In view of Corollary \ref{cor:N-in}, it suffices to ensure that the inequality \eqref{eq:N-in} holds. 

(i) Suppose first that $L_\al\le 3L_1$, hence, $L_\beta\le 3L_1$. Then, in view of the inequality $\tau\ge \la>0$, the inequality \eqref{eq:N-in} will follow from
$$
N\ge \frac{5+ 2\delta + 3L_1}{\la}. 
$$

(ii) Suppose now that $L_\al \ge 3L_1$. The function
$$
\frac{9(t+ 5 +2\delta)}{t- 2L_1}
$$
attains its maximum on the interval $[3L_1, \infty)$ at $t=3L_1$. Therefore,
$$
\frac{9(L_\al+ 5 +2\delta)}{L_\al- 2L_1}\le 27 + \frac{9(5+2\delta)}{L_1}. 
$$
Thus, the inequality
$$
\tau\ge \frac{L_\al -2L_1}{9}
$$
implies that for any 
$$
N\ge 27 + \frac{9(5+2\delta)}{L_1},
$$
we have $N\tau\ge L_\alpha + 5+2\delta$. \qed

\medskip 
Consider now the remaining case when for $\eta_0:= \frac{1}{100}\varepsilon(n,\kappa)$, the subinterval 
$\beta_0$ is empty, i.e. $\eta >\eta_0= \frac{1}{100}\varepsilon(n,\kappa)$.  Then, as above, the length $L_\al$ of the segment $x_-x_+$ is at most $2L_1$. Therefore, similarly to the case (i) 
of Corollary \ref{cor:dychotomy}, in order for $N$ to satisfy the inequality \eqref{eq:N-in}, it suffices to get 
$$
N\ge \frac{5+ 2\delta + 3L_1}{\la}. 
$$

To conclude:

\begin{thm}\label{thm:Case1}
Suppose that $g, h$ are hyperbolic isometries of $X$ generating a discrete nonelementary subgroup, whose translation lengths are equal to some $\tau\ge \la>0$. 
Let $L_1$ be such that
$$
\sinh(L_1) \sinh\left(\frac{1}{100}\varepsilon\right)=1, 
$$
where $\varepsilon= \varepsilon(n,\kappa)$. 
Then for every
\begin{equation}\label{eq:N}
N \ge \max \left( \frac{5+ 2\delta + 3L_1}{\la},   27 + \frac{9(5+2\delta)}{L_1}  \right)
\end{equation}
the group generated by $g^N, h^N$ is free of rank 2. 
\end{thm}

We note that proving that (some powers of) $g$ and $h$ generate a free subsemigroup, is easier, see \cite{BCG} and \cite[section 11]{BF}.

\begin{corollary}
\label{schottky group}
Given $g,h$ as in Theorem \ref{thm:Case1}, and any $N$ satisfying (\ref{eq:N}), the free group $\Ga_N=\langle g^{N}, h^{N}\rangle$ is  convex-cocompact. 

\end{corollary} 

\begin{proof}
Let $\mathcal{U}^{\pm }=H(g^{\pm N}x, x)$ and $\mathcal{V}^{\pm }=H(h^{\pm N}y, y)$. Observe that 
$$
g^{\pm N}(X\setminus \mathcal{U}^{\mp })\subset \mathcal{U}^{\pm }$$ 
and 
$$
h^{\pm N}(X\setminus \mathcal{V}^{\mp })\subset \mathcal{V}^{\pm }.
$$
We let $\mathfrak{D}_{g^N}, \mathfrak{D}_{h^N}$ denote the closures in $\bar{X}$ of the domains 
$$
X\setminus (\mathcal{U}^{-}\cup \mathcal{U}^+), \quad X\setminus (\mathcal{V}^{-}\cup \mathcal{V}^+)
$$
respectively and set 
\[
\mathfrak{D} = \mathfrak{D}_{g^N} \cap \mathfrak{D}_{h^N}.
\]
It is easy to see (cf. \cite{Maskit}) that this  intersection is a fundamental domain for the 
action of $\Ga_N$ on the complement $\bar{X}\setminus \Lambda$ to its limit set $\Lambda$. Therefore,  
$(\bar{X}\setminus \Lambda)/\Ga_N$ is compact. 
Hence, $\Ga_N$ is convex-cocompact (see \cite{Bo2}).
\end{proof}

\begin{rem}
It is also not hard to see directly that  the orbit maps $\Ga_N\to \Ga_N x\subset X$ are quasiisometric embeddings by following the 
proofs in \cite[section 7]{KL1} and counting the number of bisectors crossed by geodesics connecting points in $\Ga x$.  
\end{rem}

\section{Case 2: Displacement bounded above} \label{sec:Case2} 

The strategy in this case is to find an element $g'$ conjugate to $g$ (by some uniformly bounded power of $f$) 
   such that the Margulis regions of $g, g'$ are sufficiently far apart, i.e. are at distance $\ge L$, where $L$ is given by the local-to-global principle for piecewise-geodesic paths in $X$, 
   see Proposition \ref{free group 1}. 

\begin{proposition}
\label{produce large distance}
There exists a function 
$$
{\mathfrak k}: [0,\infty) \times (0, \varepsilon] \rightarrow \mathbb{N}$$ 
for $0< \varepsilon\le \varepsilon(n,\kappa)$
with the following property: Let $g_{1}, \cdots , g_{k}$ be nonelliptic 
isometries of the same type (hyperbolic or parabolic) with translation lengths $\leq \varepsilon/10$
and 
$$
k\ge {\mathfrak k}(L,\varepsilon). 
$$
Suppose that $\langle g_{i}, g_{j} \rangle$ are nonelementary discrete subgroup for all $i\neq j$. Then,   
there exists a pair of indices $i, j \in\{1,\dots,k\}$, $i\ne j$ such that 
$$d(\Hull(\TT(g_{i})), \Hull(\TT(g_{j})))>L. $$  
\end{proposition}
\proof If all the isometries $g_i$ are parabolic, then the proposition is established in 
\cite[Proposition 8.3]{KL1}. Therefore, we only consider the case when all these isometries are hyperbolic. Our proof follows closely the 
proof of  \cite[Proposition 8.3]{KL1}. 

Since for all $i\ne j$ the subgroup $\langle g_{i}, g_{j} \rangle$ is a discrete and nonelementary, and 
$\varepsilon\le \varepsilon(n,\kappa)$, we have   
$$
\mathcal{T}_{\varepsilon}(g_{i})\cap \mathcal{T}_{\varepsilon}(g_{j})=\emptyset.
$$ 
Given $L>0$, suppose
that
$$
d(\textup{Hull}(\mathcal{T}_{\varepsilon}(g_{i})), \textup{Hull}(\mathcal{T}_{\varepsilon}(g_{j})))\leq L, \quad \forall i,j\in\{1,\dots,k\}
$$ 
Our goal is to get a uniform upper bound of $k$.

Consider the $L/2$-neighborhoods $\bar{N}_{L/2}(\textup{Hull}(\mathcal{T}_{\varepsilon}(g_{i})))$. They are convex in $X$, and have nonempty pairwise intersections. Thus, by \cite[Proposition 8.2]{KL1}, there exists a point $x\in X$ such that 
$$d(x, \mathcal{T}_{\varepsilon}(g_{i}))\leq R_{1}: = n\delta +L/2+\mathfrak{q}, \quad i=1, \dots, k, $$
where $\delta$ is the hyperbolicity constant of $X$ and ${\mathfrak q}$ is as in Proposition \ref{prop:starhull}. 
Then 
$$\mathcal{T}_{\varepsilon}(g_{i})\cap B(x, R_{1})\neq \emptyset , \quad i=1, \dots, k. $$

For each $i = 1, \dots, k$ take a point $x_{i}\in \mathcal{T}_{\varepsilon}(g_{i})\cap B(x, R_{1})$ satisfying  $d(x_{i}, g_{i}^{p_i}(x_{i}))=\varepsilon$ for some $0<p_i \leq m_{g_{i}}$. Since the translation lengths of the elements $g_{i}^{p_i}$ are $\leq \varepsilon/10$, by Corollary \ref{margulis distance} there exist points $y_{i}\in X$ such that 
$$
d(y_{i}, g_{i}^{p_i}(y_{i}))= \varepsilon/3, \quad d(x_{i}, y_{i})\leq \mathfrak{r}(\varepsilon).
$$
Consider the $\varepsilon/3$-balls $B(y_{i}, \varepsilon/3)$. Then  $B(y_{i}, \varepsilon/3)\subset \mathcal{T}_{\varepsilon}(g_{i})$ since 
$$
d(z, g_{i}^{p_i}(z))\leq d(z, y_{i})+d(y_{i}, g_{i}^{p_i}(y_{i}))+d(g_{i}^{p_i}(y_{i}), g_{i}^{p_i}(z))\leq \varepsilon
$$ 
for any point $z\in B(y_{i}, \varepsilon/3)$. Thus, the balls $B(y_{i}, \varepsilon/3)$ 
are pairwise disjoint. Observe that $B(y_{i}, \varepsilon/3)\subset B(x, R_{2})$ where $R_{2}=R_{1}+\mathfrak{r}(\varepsilon)+\varepsilon/3$. 

Let $V(r, n)$ denote the volume of the $r$-ball in $\mathbb{H}^{n}$. Then for each $i$, $\mathrm{Vol}(B(y_{i}, \varepsilon/3))$ 
is at least $V(\varepsilon/3, n)$, see \cite[Proposition 1.1.12]{Bo2}. Moreover, the volume of $B(x, R_{2})$ is at most $V(\kappa R_{2}, n)/ \kappa^{n}$, see \cite[Proposition 1.2.4]{Bo2}. Let 
$$
{\mathfrak k}(L, \varepsilon):=\dfrac{V(\kappa R_{2}, n)/ \kappa^{n}}{V(\varepsilon/3, n)}+1.$$
Then $k< {\mathfrak k}(L, \varepsilon)$, because otherwise we would obtain
 $$
\mathrm{Vol}\left( \bigcup_{i=1}^k B(y_{i}, \varepsilon/3)\right) > \mathrm{Vol}(B(x, R_{2})),
 $$
 where the union of the balls on the left side of this inequality is contained in 
 $B(x, R_{2})$, which is a contradiction. 
 
  Therefore, whenever $k\geq {\mathfrak k}(L, \varepsilon)$, there exist a pair of indices $i, j$ such that 
  \[
  d\left(\Hull(\TT(g_{i})), \Hull(\TT(g_{j}))\right)>L. \qedhere
  \]

\begin{rem}
Proposition \ref{produce large distance} also holds for 
isometries of mixed types (i.e. some $g_i$'s are parabolic and some are hyperbolic). The proof is similar to the one given above. \end{rem}

\begin{theorem}\label{thm:Case2}
For every nonelementary discrete  subgroup $\Ga= \langle g, h \rangle <\Isom(X)$ with 
$g, h$ nonelliptic isometries  satisfying  
$$
\tau(g)\leq \varepsilon/10\le \varepsilon(n,\kappa)/10, 
$$ 
there exists $i$, $ 1 \le i \le  \mathfrak{k}(L(\varepsilon/10), \varepsilon)$, such that  $\langle g, h^{i}gh^{-i}\rangle$ 
is a free subgroup of rank 2, where ${\mathfrak k}$ is the function given by Proposition \ref{produce large distance} and $L(\varepsilon/10)$ is the constant in Proposition \ref{qua}.

\end{theorem}
\begin{proof} Consider isometries $g_i:=h^{i}gh^{-i}$, $i\ge 1$. We first claim that no 
pair $g_i, g_j$, $i\ne j$, generates an elementary subgroup of $\Isom(X)$. There are two cases to consider:

\medskip
(i) Suppose that $g$ is parabolic   with the  fixed point $p\in \geo X$. We claim that for  all  
$i\ne j$,  $h^{i}(p)\neq h^{j}(p)$. Otherwise, $h^{j-i}(p)=p$, and $p$ would be a fixed point of $h$. But this would imply that $\Ga$ is elementary, contradicting our hypothesis. 

\medskip 
(ii) The proof in the case when $g$ is hyperbolic is similar. The axis of $g_i$ equals 
$h^{i}(A_{g})$. If hyperbolic isometries $g_i, g_j$, $i\ne j$, generate a discrete elementary subgroup of $\Ga$, then they have to share the axis, and we would obtain
$h^{i}(A_{g})=h^{j}(A_{g})$. Then $h^{j-i}(A_{g})=A_{g}$. Since $h^{j-i}$ is nonelliptic, it 
cannot swap the fixed points of $g$, hence, it fixes both of these points. Therefore, $g, h$ 
have common axis, contradicting  the hypothesis that $\Ga$ is nonelementary.

\medskip
All the isometries $g_i$ have equal translation lengths $\leq \varepsilon/10$. Therefore, 
by Proposition \ref{produce large distance}, there exists a pair of natural numbers 
$i, j\leq {\mathfrak k}(L(\varepsilon/10), \varepsilon)$ such that 
$$
d(\Hull(\TT(h^{i}gh^{-i})), \Hull(\TT(h^{j}gh^{-j})))> L(\varepsilon/10)$$
where ${\mathfrak k}(L(\varepsilon/10), \varepsilon)$ is the function as in Proposition \ref{produce large distance}. 
It follows that 
$$
d(\Hull(\TT(h^{j-i}gh^{i-j})), \Hull(\TT(g)))> L(\varepsilon/10).
$$
Setting  $f:=h^{j-i}gh^{i-j}$, and applying Proposition \ref{free group 1} to the isometries 
$f, g$, we conclude that the subgroup  $\langle f, g \rangle <\Ga$ is  free of rank $2$.
The word length of $f$ is at most 
\[
2|j-i|+1 \leq 2{\mathfrak k}(L(\varepsilon/10), \varepsilon) +1. \qedhere
\]
\end{proof}

\section{Conclusion}

Now we are in a position to complete the proof of Theorem \ref{thm:main}. 

\begin{proof}[Proof of Theorem \ref{thm:main}]
We set $\la:= \varepsilon/10$, where $\varepsilon = \varepsilon(n,\kappa)$ is the Margulis constant. 
Let $g, h$ be non-elliptic isometries of $X$ generating a discrete nonelementary subgroup of $\Isom(X)$, such that $\tau(g)=\tau(h)=\tau$.

If $\tau\ge \la$, then by Theorem \ref{thm:Case1}, the subgroup 
$\Ga_N<\Ga$ generated by $g^N, h^N$ is free of rank $2$, where 
$$
N :=  \left\lceil \max\left( \frac{5+ 2\delta + 3L_1}{\lambda},  
 27 + \frac{9(5+2\delta)}{L_1} \right) \right\rceil. 
$$
Here $\delta=\cosh^{-1}(\sqrt{2})$, and 
$$
L_1=\sinh^{-1}\left(  \frac{1}{\sinh (\varepsilon/100)} \right).
$$

If $\tau\le \la$, then by Theorem \ref{thm:Case2} there exists $i\in [1, \mathfrak{k}(L(\lambda), \varepsilon)]$ such that 
 $\langle g, h^i g h^{-i}\rangle$ is free of rank $2$ where $\mathfrak{k}(L(\lambda), \varepsilon)$ is a constant as in Theorem \ref{thm:Case2}.
\end{proof}

\end{document}